\documentclass{ifacconf}

\usepackage{graphicx}      % include this line if your document contains figures
\usepackage{natbib}        % required for bibliography

\usepackage{amsmath,amssymb,amsfonts}
\usepackage{algorithmic}
\usepackage{graphicx}
\usepackage{textcomp}
\usepackage{amsopn, amstext,color}

\newtheorem{problem}{Problem}[section]

\usepackage{subfigure}
\newtheorem{remark}{Remark}

\begin{document}
\begin{frontmatter}

\title{The distance-based formation controller design for multi-agent systems in port-Hamiltonian form\thanksref{footnoteinfo}} 

\thanks[footnoteinfo]{This work was supported by the ERC Starting Grant iSwarm
	$101076091$, and the National Natural Science Foundation of China under Grants $U24A20263$ and $62595804$.}

\author[third]{Jingyi Zhao} 
\author[Second]{Yongxin Wu} 
\author[third]{H\'ector Garc\'ia de Marina}
\author[First]{Yuhu Wu}
\author[Second]{Yann Le Gorrec}

\address[third]{Department of Computer Engineering, Automation, and Robotics, and with IMAG, University of Granada, 18001, Granada, Spain. (e-mail: jingyizhao@ugr.es; hgdemarina@ugr.es)
}
\address[Second]{Universit\'e Marie et Louis Pasteur, SUPMICROTECH, CNRS, institute FEMTO-ST, F-25000, Besan\c{c}on, France (e-mail: yongxin.wu@femto-st.fr; yann.le.gorrec@ens2m.fr)
}
\address[First]{The Key Laboratory of Intelligent Control and Optimization for Industrial Equipment of Ministry of Education and the School of Control Science and Engineering, 
	Dalian University of Technology, Dalian, 116024, China (e-mail: wuyuhu@dlut.edu.cn).
	}

\begin{abstract}                % Abstract of 50--100 words
%Based on the practical scenario where collisions in formation control may lead to agent damage, this paper addresses the distance-based formation and collision avoidance problem for multi-agent systems with port-Hamiltonian dynamics. Under the assumption that each agent can obtain the positions of all other agents via communication networks or sensors, an incidence matrix corresponding to edges and nodes is first constructed by assigning a specific sign to each node within the multi-agent system. Under the proposed framework, the resulting incidence matrix is associated with a directed acyclic graph, which exhibits the favorable property of being full column rank. Subsequently, an improved artificial potential field method is designed, featuring an attraction-only potential field, while collision avoidance is ensured through the introduction of safety barrier constraints. This design effectively avoids the local minima issue commonly encountered in conventional artificial potential field methods, where attraction and repulsion forces may counteract each other. Furthermore, a unified controller integrating velocity tracking, formation maintenance, and collision avoidance is proposed. The stability of the closed-loop system is rigorously analyzed using LaSalle's invariance principle. Numerical simulations are provided to demonstrate the effectiveness of the proposed control strategy.
Based on the practical scenario where collisions in formation control may lead to agent damage, this paper investigates the integrated problem of distance-based formation control and collision avoidance for multi-agent systems governed by port-Hamiltonian dynamics. A foundational step involves constructing a signed incidence matrix, which, by design, corresponds to a directed acyclic graph and possesses the full column rank property. To overcome the prevalent issue of local minima in traditional artificial potential fields, a novel design utilizing attraction-only potentials is introduced, with collision avoidance rigorously enforced by safety barriers. This framework leads to a unified controller that concurrently manages velocity tracking, target formation acquisition, and inter-agent safety. The stability of the resulting closed-loop system is guaranteed through LaSalle's invariance principle. Numerical simulations demonstrate the validity and effectiveness of the proposed control strategy.
\end{abstract}


\end{frontmatter}
%===============================================================================

\section{Introduction}
Formation control, a significant category of cooperative control for multi-agent systems, aims to enable a group of agents to achieve and maintain a specific geometric configuration through distributed strategies \citep{r8}. This technology finds broad applications in domains \citep{r1}, such as cooperative unmanned aerial vehicle reconnaissance \citep{r7} and collaborative manipulation by mobile robots \citep{r9}. Formation control methodologies can be primarily categorized into position-based, distance-based, and bearing-based approaches. Among them, distance-based formation control, which relies solely on maintaining relative distances between agents and is independent of a global coordinate frame, offers enhanced robustness and flexibility \citep{AS}.

In practical dynamic environments, ensuring collision avoidance among agents during the formation process presents a critical safety challenge that must be addressed. %Collision avoidance is typically addressed through the design of repulsive potential fields, the establishment of safety zones, or the introduction of optimization constraints. 
Under this background, the traditional artificial potential field method is widely adopted due to its computational simplicity. However, the inherent local minimum problem, where the system converges to a non-target state when attractive and repulsive forces balance, severely limits its reliability in complex missions \citep{r15}. While existing research presented many methods to address this limitation, incorporating random walks, navigation functions, or model predictive control, these solutions often lack theoretical stability guarantees or suffer from high computational complexity, hindering real-time application \citep{r14}.

Meanwhile, the port-Hamiltonian (PH) framework provides a powerful energy-based theory for the modeling and control of multi-agent systems. This framework explicitly represents systems as structures of energy storage, dissipation, and interconnection, making passivity-based control design natural and intuitive \citep{r2}. In recent years, some works have successfully applied the PH framework to synchronization and formation control of multi-agent systems, leveraging its inherent stability properties to simplify analysis, such as \citep{Nb1,r6,r5,r4,r10,r11}. However, integrating collision avoidance constraints within this framework and fundamentally resolving the local minimum issue of traditional potential fields remains an open research question. 
In addition, to ensure a unique distance-based formation for multi-agent systems, the Minimally and Infinitesimally Distance Rigid (MIDR) framework is widely adopted for its efficiency—prescribing distances over exactly $nN-2n(n+1)$ edges to minimize inter-agent interactions. However, a critical limitation arises in collision avoidance scenarios: the MIDR framework only enforces distance constraints along its predefined minimal edges, leaving non-adjacent agent pairs unregulated. This point leads to potential collisions during formation convergence, as the framework fails to guarantee all agent pairs maintain a safe threshold distance.

Motivated by this problem, this paper investigates the integrated distance-based formation control and collision avoidance problem for multi-agent systems with PH dynamics. The contributions of this work are threefold: Firstly, by assigning specific signs to nodes, an incidence matrix corresponding to a tournament acyclic graph is constructed, which possesses the full column rank property, thereby establishing a graph-theoretic foundation for the formation control. Secondly, an improved artificial potential field utilizing only attractive forces is proposed, combined with safety barrier certificates to ensure collision avoidance between all agents, thereby preventing the occurrence of local minimum. Finally, a unified distributed controller integrating velocity tracking, formation, and collision avoidance is developed, and the stability of the closed-loop system is analyzed. 

The remainder of this paper is structured as follows. Section 2 provides the problem formulation and necessary preliminaries. In Section 3, a distance-based formation controller with collision avoidance is presented. Section 4 is devoted to numerical simulations, which are conducted to validate the effectiveness and advantages of the proposed controller over the existing one. Finally, Section 5 concludes this paper.
\section{Problem Formulation}
In this subsection, the distance-based formation control problem with collision avoidance is formulated.
\subsection{The incidence matrix of the graph}
Consider $N$ agents communicating with a directed graph $\mathcal{G}=\{\mathcal{V},\mathcal{E}\}$, where $\mathcal{V}=\{1,\cdots,N\}$ denotes the node set, $\mathcal{E}=\{E_1,\cdots,E_M\}$ 
denotes the edge set, and ${E_k}$ denotes the edge between agents $i$ and $j$. 
The incidence matrix $B=(b_{ij})_{N\times M}$ of the graph $\mathcal{G}$ describes the relationship between nodes and edges, where
\begin{equation*}
	b_{ij}=\begin{cases}
		+1, \quad\text{if node $i$ is the tail (source) of edge $E_k$,}\\
		-1,\quad\text{if node $i$ is the head (sink) of edge $E_k$,}\\
		0,	\quad\text{otherwise.}\end{cases}
\end{equation*}
A \textit{cyclic graph} is a graph that contains at least one cycle. A cycle is a closed path where the only repeated vertices are the first and last ones. Conversely, an \textit{acyclic graph} is a graph that contains no cycles. Based on the definition of the incidence matrix $B$, one can find that $B$ of an acyclic graph is full column rank.

In the \textit{complete graph} $\mathcal{G}$, there exists an edge between any two agents $i$ and $j$, and it means that each agent can obtain the information from any other agents.
%There are many different types of graphs, if 

A \textit{tournament graph} $\mathcal{G}^t$ is a directed graph obtained by assigning a direction to each edge in a complete undirected graph. In other words, it is an orientation of a complete graph, which means every pair of distinct vertices is connected by a single directed edge. More details can be found in \citep{rb}.
\subsection{The port-Hamiltonian model}
Motivated by \citep{Nb1}, to simplify the problem, we assume that each agent can be seen as an single point mass in $\mathbb{R}^n$, and the dynamic of the $i$th agent ($i\in\mathcal{V}$) is described under the PH framework as 
\begin{equation}\label{e1}
	\begin{aligned}
	&\begin{bmatrix}
		\dot q_i(t)\\ \dot p_i(t)
	\end{bmatrix}=\begin{bmatrix}
	0&I_{n}\\-I_n&-D_i
	\end{bmatrix}\begin{bmatrix}\nabla_{q_i} H_i(q_i,p_i)\\ \nabla_{p_i} H_i(q_i,p_i)	\end{bmatrix}%\begin{bmatrix}\frac{\partial H_i(q_i,p_i)}{\partial q_i}\\ \frac{\partial H_i(q_i,p_i)}{\partial p_i}\end{bmatrix}
	+\begin{bmatrix}0\\I_n
	\end{bmatrix}u_i(t),\\
	&y_i(t)= \nabla_{p_i} H_i(q_i,p_i),
	\end{aligned}
\end{equation}
where $q_i(t)\in\mathbb{R}^n$ and $p_i(t)\in\mathbb{R}^n$ denotes the position and the corresponding momentum of the $i$th agent. Based on the coupling relationship of position and momentum, we have $p_i(t)=m_i\dot q_i(t)$, where the positive constant $m_i\in\mathbb{R}$ is the mass of the $i$th agent, $\nabla_{q_i} H_i(q_i,p_i)=\frac{\partial H_i(q_i,p_i)}{\partial q_i}$, $\nabla_{p_i} H_i(q_i,p_i)=\frac{\partial H_i(q_i,p_i)}{\partial p_i}$. The control input of the $i$th system is denoted by $u_i(t)\in\mathbb{R}^n$ while the output is denoted by $y_i(t)\in\mathbb{R}^n$, and the viscous friction is modeled by a positive semi-definite dissipation matrix $D_i\in\mathbb{R}^{n\times n}$. The smooth Hamiltonian function $H_i(q_i,p_i)=\frac{1}{2}p_i(t)^\top M_i^{-1}p_i(t)$, where the matrix $M_i=m_iI_n$, $I_n\in\mathbb{R}^{n\times n}$ is an identity matrix.

\subsection{The control objective}
Based on the PH system \eqref{e1}, we define $q_{E_k}(t)=q_i(t)-q_j(t)$, then the distance between the $i$th agent and the $j$th agent is described as 
$$d_{E_k}(t)=\Vert q_{E_k}(t)\Vert =\Vert q_i(t)-q_j(t)\Vert.$$
 Next, the following control objectives are given:

(1) In distance-based formation control for multi-agent systems, a strightforward idea is to specify target distances to define the desired geometric shape. Define the desired distance in edges $E_1,\cdots, E_M$ as 
$$d_{E_1}^*,\cdots,d_{E_M}^*\in\mathbb{R}^+,$$
then the \textit{distance-based formation control objective} of the multi-agent systems is 
\begin{equation}\label{o1}
\lim_{t\rightarrow\infty} d_{E_k}(t)=d_{E_k}^*,\quad \!{E_k}\in\mathcal{E}.
\end{equation}

(2) Define the desired velocity as $v^*=v_1^*\otimes I_N\in\mathbb{R}^{nN}$, then the \textit{velocity tracking objective} is described as 
 \begin{equation}\label{o2}
 	\lim_{t\rightarrow\infty}\dot q =v^*,
 \end{equation}
 where $v_1^*\in\mathbb{R}^n$ is the desired velocity of each agent.
%Given the requirement to maintain formation, we  $v_1^*=\cdots=v_N^*\in\mathbb{R}^n$. 
%In view of the fact that the control objective of the multi-agent system is to maintain formation operation, it follows that the target velocity for each agent is designed to be consistent, i.e., $v_1^*=\cdots=v_N^*$. 

(3) In reality, having formation control goals alone is inadequate, as crashes may happen during the functioning of agents. Hence, we introduce the \textit{collision avoidance objective} of the multi-agent systems as
\begin{equation}\label{o3}
 d_{E_k}(t) \geq d_s,
\end{equation}
where $d_{s}\in\mathbb{R}$ is predefined based on the safe distance between any two agents, 
${E_k}\in\mathcal{E}$. To ensure the practicality of the problem, the target formation distance is designed to exceed the safety distance, i.e., $d_{E_k}^*>d_{s}$.

Based on the above analysis, the control problem is formulated as follows:
\begin{problem}
	Design $u(t)=\text{col}(u_1(t),\cdots,u_N(t))\in\mathbb{R}^{nN}$ for the multi-agent systems with PH dynamics \eqref{e1}, such that the control objectives \eqref{o1}-\eqref{o3} are achieved.
\end{problem}
\section{Main Results}
In this section,to resolve the collision avoidance limitation of the MIDR framework, we first propose a strongly connected acyclic tournament graph structure, and then a controller $u$ is designed for the PH multi-agent systems \eqref{e1} to achieve the control objectives \eqref{o1}-\eqref{o3}.
\subsection{The design of the incidence matrix}
As shown in Fig.~\ref{f2}, a MIDR graph only enforces distance constraints along its specified edges, leaving non-adjacent agent distances unregulated. This inherent limitation gives rise to potential collisions among unmonitored agent pairs during formation convergence. For effective collision avoidance, all agent pairs must maintain distances above a predefined safe threshold, necessitating a strongly connected graph to cover all inter-agent relationships.

However, direct adoption of a strongly connected graph for inter-agent communication undermines the favorable properties of the MIDR framework’s minimum rigid graph, invalidating Li et al.’s stability analysis approach. To reconcile strong connectivity (for collision avoidance) and theoretical rigor (for stability proof), a graph structure with matrix properties supporting subsequent derivations is essential.

Hence, we extend the MIDR framework by proposing an acyclic tournament graph $\mathcal{G}^t=(\mathcal{V},\mathcal{E})$. Its associated signed incidence matrix is column-full rank, ensuring agents' distance regulation (mitigating collision risks), while providing critical mathematical support for stability analysis. The specific design process are detailed as follows:
\begin{itemize}
	\item Agent $1$ is the tail node for the first $N-1$ edges ($E_1,\cdots,E_{N-1}$) with $2,\cdots,N$, respectively.
	\item Agent $2$ is the tail node for the $N-2$ edges ($E_N,\cdots,E_{2N-3}$) with $3,\cdots,N$.
	\item Agent $i\in\{2,\cdots,N-1\}$ is the positive node for the $N-i$ edges with $i+1,\cdots,N$.
	\item Agent $N$ is the head node for $N-1$ edges with $1,\cdots,N-1$, respectively.
\end{itemize}
An example of $\mathcal{G}^t$ with $8$ agents is given as Fig.~\ref{f1}, the corresponding incidence matrix $B\in\mathbb{R}^{8\times 28}$ is 
\begin{equation}
B= \begin{bmatrix}
	B_1&B_2&\cdots&B_7
\end{bmatrix},
\end{equation}
where
\begin{equation*}
		\begin{cases}
	B_1=
	\begin{bmatrix}
		\mathbf{1}_7^T \\
		-I_7
	\end{bmatrix}\in\mathbb{R}^{8\times 7}, \\B_2=
	\begin{bmatrix}
		\mathbf{0}_6^T \\
		\mathbf{1}_6^T \\
		-I_6
	\end{bmatrix}\in\mathbb{R}^{8\times 6},\\
	%B_3=\begin{bmatrix}\mathbf{0}_{2\times5}\\\mathbf{1}_5^T \\	-I_5\end{bmatrix}\in\mathbb{R}^{8\times 5},\\ 
\qquad	\vdots\\
	B_7=\text{col}(0,0,0,0,0,0,1,-1)\in\mathbb{R}^8. 
	\end{cases}
\end{equation*} 
We can find that the designed incidence matrix of the acyclic tournament graph $\mathcal{G}^t$ has full column rank.
 \begin{figure}
	\begin{center}
		\includegraphics[width=4.5cm]{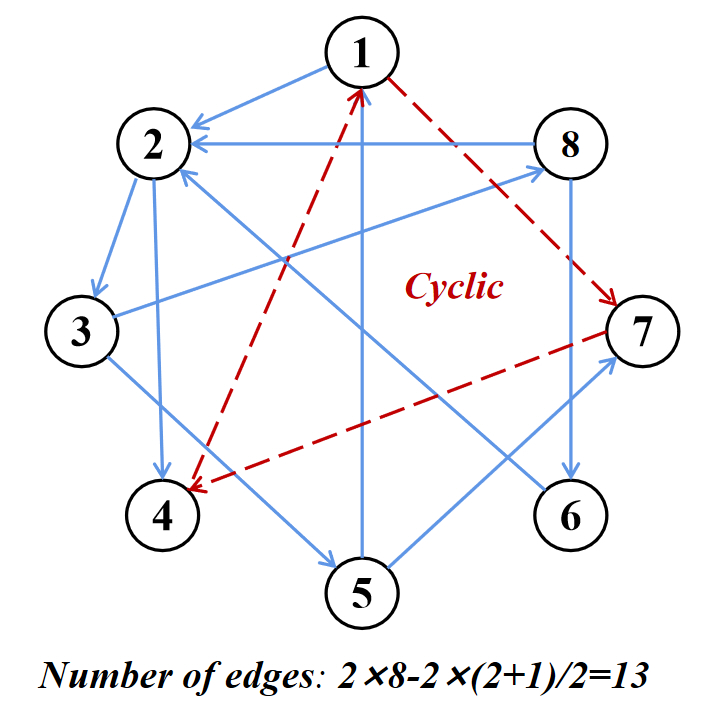}    % The printed column width is 8.4 cm.
		\caption{An example of the graph in MIDR framework with $8$ agents.} 
		\label{f2}
	\end{center}
\end{figure}
\begin{figure}
\begin{center}
\includegraphics[width=4.5cm]{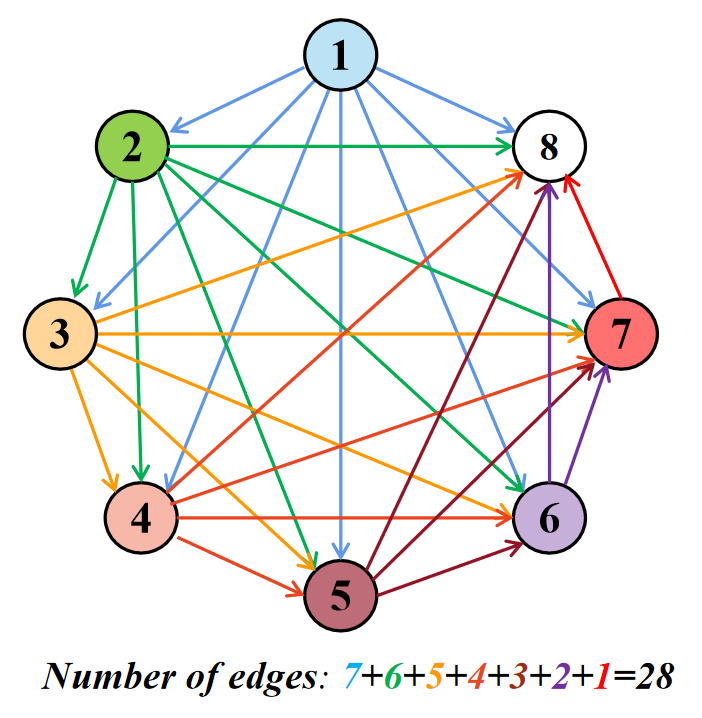}    % The printed column width is 8.4 cm.
\caption{An example of tournament graph with $8$ agents.} 
\label{f1}
\end{center}
\end{figure}
\begin{remark}
This section merely proposes a method for specifying edge directions to define the association matrix, which does not affect the graph's connectivity. In other words, the communication topology graph describes how each agent can access information, while node polarity is solely for controller design in the next section and does not impact information exchange between agents.
\end{remark}
\subsection{The distance-based formation controller design}
In this subsection, the distance-based formation controller is designed for multi-agent systems. At first, we design the velocity tracking controller as 
\begin{equation}\label{u1}
	u_i^v=-D_i v_1^*-D_i^v\frac{\partial H^v}{\partial \bar p_i},
\end{equation}
where the momentum error $\bar p_i=p_i-M_iv_1^*$, the positive semi-definite matrix $D_i^v\in\mathbb{R}^{n\times n}$, and the corresponding Hamiltonian function is 
\begin{equation}\label{h1}
		H^v=\frac{1}{2}\sum_{i=1}^{N}\Vert M_i^{-1}p_i-v_1^*\Vert^2
		=\frac{1}{2}\sum_{i=1}^{N}\Vert M_i^{-1}\bar p_i\Vert^2,
\end{equation}
By inserting controller \eqref{u1} into PH system \eqref{e1} and choose $H^v$ as a Lyapunov function, then we can prove that the velocity tracking objective can be achieved.

Next, we design the formation controller by assigning a virtual spring and damping coupling in each edges. For the collision avoidance objective, traditional standard potential field methods have certain limitations, the most significant of which is illustrated in Fig.~\ref{f3}.
\begin{figure}
	\begin{center}
		\includegraphics[width=5.5cm]{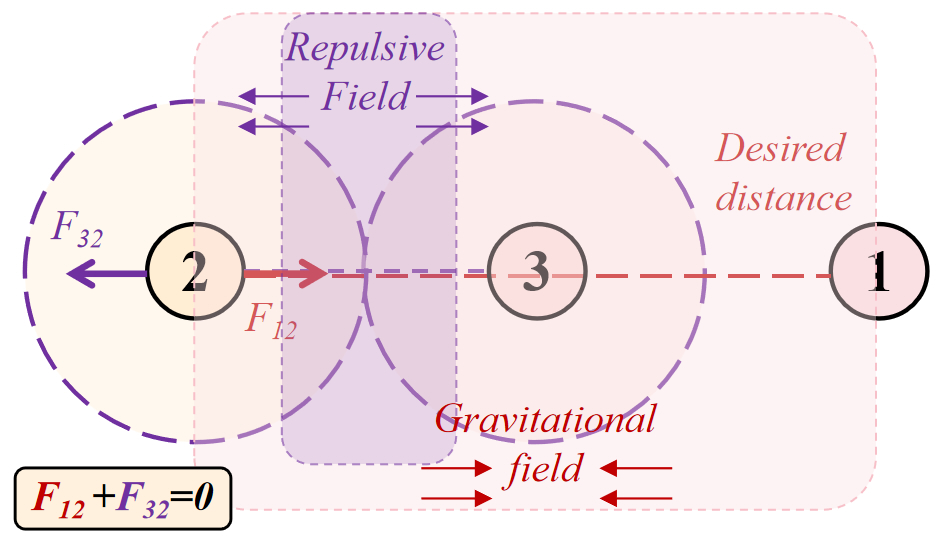}    % The printed column width is 8.4 cm.
		\caption{The cancellation of attractive and repulsive forces in traditional manual potential field methods.} 
		\label{f3}
	\end{center}
\end{figure}
The core concept of the artificial potential field method is to construct a virtual potential field within the agent's operational space. Within this field, target points generate gravitational forces that attract agents toward them, while obstacles produce repulsive forces that deter agents from approaching them. The movement of each agent is determined by its resultant force. This mechanism is the cause of the system's settling into a local minimum, where the resultant force is equal to zero when the repulsive and attractive forces are equal in magnitude and opposite in direction.

By denoting the error variable $e_k=\Vert q_{E_k}\Vert^2-(d_{E_k}^*)^2$ as the difference between the actual and the desired distance on the $k$th edge, the positive constants $\alpha_k$, $k=1,\cdots,M^t$ are chosen as weighting factors, where $M^t=\frac{N(N-1)}{2}$ denotes the total number of edges.  By combining the virtual spring and artificial potential
field methods, we design the Hamiltonian function for the $k$th edge to achieve the formation and collision avoidance as
$$H_k^f=\frac{1}{4}\alpha_k\left(\frac{1}{e_k+(d_{E_k}^*)^2-d_{s}^2}\! -\! \frac{1}{(d_{E_k}^*)^2-d_{s}^2}\right)^2.$$  
The following control law is designed to implement the virtual spring–damper coupling on each edge:
\begin{equation}\label{eq:controlbyinterconnection_k}
	\begin{aligned}
		\dot e_k&=\delta_k,\\
		\xi_k^c&=\frac{\partial H_k^f}{\partial e_k}+D_k^c\delta_k,
	\end{aligned}
\end{equation}
where $\delta_k, \xi_k^c\in\mathbb{R}$ denote the controller input and output, respectively. The dissipation parameter $D_k^c\in\mathbb{R}$ is positive.  

%By combining the virtual spring and artificial potential field methods, we design the Hamiltonian function for the  formation and collision avoidance objective as 
%\begin{equation}
%	\begin{aligned}
%		H^f\!\!=\!\!\frac{1}{4}\sum_{k=1}^{M^t}\alpha_k\left(\frac{1}{e_k\!+\!(d_{E_k}^*)^2-d_{s}^2}\!-\!\frac{1}{(d_{E_k}^*)^2-d_{s}^2}\right)^2,
%	\end{aligned}
%\end{equation}
%where the error $e_k=\Vert q_{E_k}\Vert^2-(d_{E_k}^*)^2$ denotes the difference between the actual and the desired distance on the $k$th edge, the positive constants $\alpha_k$, $k=1,\cdots,M^t$ are chosen as the weighting factor, the edge's number $M^t=\frac{N(N-1)}{2}$.
%By assigning a virtual spring and damping coupling in each edges, we have
% and the Hamiltonian function in the $k$th edge is 
%$$H_k^f=\frac{1}{4}\alpha_k\left(\frac{1}{e_k+(d_{E_k}^*)^2-d_{s}^2}\! -\! \frac{1}{(d_{E_k}^*)^2-d_{s}^2}\right)^2.$$
Since the output of the controller is $\xi_k^c\in\mathbb{R}$, the following mapping Jacobian is introduced to transform it into $\mathbb{R}^n$,
$$\mathcal{J}_k=\frac{\partial e_k}{\partial q_{E_k}}=q_{E_k}.$$ 
Accordingly, the corresponding formation and collision avoidance control input is designed as
\begin{equation}\label{uc1}
	u_k^c=-\mathcal{J}_k \xi_k^c=-q_{E_k}\left(\frac{\partial H_k^f}{\partial e_k}+D_k^c\dot e_k\right),
\end{equation}
By stacking the controllers for all $M^t$ edges, we obtain 
\begin{equation}\label{eq:controlbyinterconnection}
	\begin{aligned}
		\dot e&=\delta,\\
		\xi^c&=\frac{\partial H^f}{\partial e}+D^c\delta,
	\end{aligned}
\end{equation}
where  $\delta = \begin{bmatrix}
\delta_1,\cdots ,  \delta_{M^t}
\end{bmatrix}$,  $\xi^c = \begin{bmatrix}
\xi^c_1, \cdots,  \xi^c_{M^t}
\end{bmatrix}$ and $D^c=\text{diag}(D^c_1,\cdots,D^v_{M^t})$. The Hamiltonian function is defined as
\begin{equation}
	\begin{aligned}
		H^f\!\!=\!\!\frac{1}{4}\sum_{k=1}^{M^t}\alpha_k\left(\frac{1}{e_k\!+\!(d_{E_k}^*)^2-d_{s}^2}\!-\!\frac{1}{(d_{E_k}^*)^2-d_{s}^2}\right)^2.
	\end{aligned}
\end{equation}
The controller \eqref{eq:controlbyinterconnection} is interconnected with the multi-agent systems through the incidence matrix $B$ of the tournament graph $\mathcal{G}^t$ via the following power preserving relation
\begin{equation}
\begin{bmatrix}
u^c\\\delta
\end{bmatrix}= \begin{bmatrix}
0 &  -(B\otimes I_n)\Omega^\top\\ \Omega(B\otimes I_n)^\top & 0
\end{bmatrix}\begin{bmatrix}
\bar y\\\xi^c
\end{bmatrix}
\end{equation}
where the matrix   $\Omega=\text{diag}(q_{E_1},\cdots, q_{E_{M^t}})$, and $\bar y=M^{-1}p-v^*$. 
By combining the velocity tracking controller \eqref{u1} and the formation controller \eqref{uc1}, the overall control input can be written as:
\begin{equation}\label{u2}
u^f = u^v+u^c.
\end{equation}
Here, the first term presents a damping injection \eqref{u1}, while the second term corresponds to the energy-shaping term generated by the dynamic controller \eqref{eq:controlbyinterconnection}. Consequently, the total Hamiltonian function of the closed loop system is given by: 
$$H^F=H^v+H^f.$$ 
%The overall control scheme is illustrated in Fig.\ref{fc}. 
%\begin{figure}[!htbp]
%	\begin{center}
%		\includegraphics[width=9cm]{fh6}   
%		\caption{The control scheme of the PH system.} 
%		\label{fc}
%	\end{center}
%\end{figure}
%As shown in Fig.\ref{fc}, multi-agent systems utilize communication topology maps for positional interaction or detect the positions of other agents through sensors. Then, an error feedback controller and a velocity tracking controller designed based on the distance between agents are employed to achieve formation control.

%By inserting $u^f$ in \eqref{u2} into the PH system \eqref{e1}, we obtain the closed-loop system as
%\begin{equation}\label{cc1}
%	\begin{bmatrix}
%		\dot {q}\\ \dot{p}
%	\end{bmatrix}\!\!=	\!\!
%	\begin{bmatrix}
%		0 & I_{nN}\\ -I_{nN} & -D^C\\
%	\end{bmatrix}\begin{bmatrix}\nabla_q H^F\\ \nabla_p H^F	\end{bmatrix}+ \begin{bmatrix} 0\\(B\otimes I_n)\Omega^\top F(e)\\ 
%	\end{bmatrix},
%\end{equation}
%where the dissipation matrix $D^C=D+D^v+(B\otimes I_n)^\top \Omega D^c\Omega^\top(B\otimes I_n)$, the mapping $F_e=\nabla_e H^f$ with the error vector $e=\text{col}(e_1,\cdots,e_{M^t})$, and 

The following theorem illustrates the stability of the closed-loop systems under the action of the proposed controller \eqref{u2}.
\begin{thm}   % use the thm environment for theorems
Consider $N$ agents modeled under the PH framework as described by \eqref{e1}, and assume the initial distance between any two agents satisfies $\Vert q_{E_k}(0)\Vert>d_s$, then under the action of the proposed control law $u^f$ in \eqref{u2}, the multi-agent systems asymptotically converge to the desired formation without collision.
\end{thm}

\begin{pf}   
First, we show the asymptotically stability of the system  \eqref{e1} with the control law \eqref{u2}.   Design the position error as $\bar q=q-v^*t$, then the closed loop system can be rewritten as follows:
\begin{equation}\label{c1}
\!\!	\begin{bmatrix}
			\dot {\bar q}\\ \dot{\bar p}\\ \dot e
		\end{bmatrix}\!\!=\!\!
		\begin{bmatrix}
			0 & I_{nN} & 0\\ -I_{nN} & -D^C & -(B\otimes I_n)\Omega^\top \\
			0&\Omega(B\otimes I_n)^\top&0\\
		\end{bmatrix}\!\!\!\!
		\begin{bmatrix}\nabla_{\bar q} H^F\\ \nabla_{\bar p} H^F\\ \nabla_{e} H^F	\end{bmatrix},
\end{equation}
where the dissipation matrix $D^C=D+D^v+(B\otimes I_n)^\top \Omega D^c\Omega^\top(B\otimes I_n)$, 
$D=\text{diag}(D_1,\cdots,D_N)$, $D^v=\text{diag}(D^v_1,\cdots,D^v_N)$. Then we only need to prove the error system satisfies $\lim_{t\rightarrow\infty}e=\textbf{0}_{M^t}$, $\lim_{t\rightarrow\infty}\bar p=\textbf{0}_{nN}$. 
Take the Hamiltonian function $H^F$ as the candidate Lyapunov function, then we have $H^F\geq 0$, and the derivative of $H^F$ is
$$\dot H^F=-(\frac{\partial H^c}{\partial \bar p})^\top D^C\frac{\partial H^c}{\partial \bar p}=-\bar p^\top M^{-1}D^DM^{-1} \bar p\leq 0.$$
Hence, by LaSalle's Invariance Principle, the closed-loop system \eqref{c1} asymptotically converge to the largest invariance set $\{\bar p\mid\bar p=\textbf{0}_{nN}\}$.
By inserting $\bar p=\textbf{0}_{nN}$, $\dot {\bar p}=\textbf{0}_{nN}$ into \eqref{c1}, we have 
$$(B\otimes I_n)\Omega^\top F(e)=\textbf{0}_{nN},$$
where $F(e)=\text{col}(f_1(e_1)g_1(e_1),\cdots,f_{M^t}(e_{M^t})g_{M^t}(e_{M^t}))\in\mathbb{R}^{M^t}$ with $f_k(e_k):\mathbb{R}\rightarrow\mathbb{R}$ and $g_k(e_k):\mathbb{R}\rightarrow\mathbb{R}$ of $e_k$, $k=1,\cdots,M^t$ as
\begin{equation}
	\begin{aligned}
		f_k(e_k)&=-\frac{2\alpha_k}{\big(e_k+(d_{E_k}^*)^2-d_s^2\big)^2},\\
		g_k(e_k)&=(\frac{1}{e_k+(d_{E_k}^*)^2-d_s^2}-\frac{1}{(d_{E_k}^*)^2-d_s^2})^2, 
	\end{aligned}
\end{equation}

Based on the design approach of the tournament topology, we conclude that $\mathcal{G}^t$ is an acyclic graph. It means the incidence matrix $B$ of $\mathcal{G}^t$ has the linearly independent columns, then for each $k$, by $\alpha_k>0$, we have
$$q_{E_k}^\top f_k(e_k)g_k(e_k)=\textbf{0}_n,$$
which implies that $q_{E_k}=\textbf{0}_n$ or $f_k(e_k)=0$ or $g_k(e_k)=0$ for $k=1,\cdots,M^t$ holds. Clearly, there is no zeros for the function $f_k$. By analyzing the zero of function function $g_k$, we obtain 
\begin{equation}
	\begin{aligned}
%		f_k(e_k)&=0\iff e_k+(d_{E_k}^*)^2=0,\\
		g_k(e_k)=0\iff e_k=0.
	\end{aligned}
\end{equation}
Noticing that $\Vert q_{E_k}\Vert^2=e_k+(d_{E_k}^*)^2$, we have for $k=1,\cdots,M^t$, we have the closed-loop system \eqref{c1} asymptotically converges to the set
$$S=\{(\bar q,\bar p,e)\vert \bar p=\textbf{0}_{nN},  \ ( e_k=0 \lor e_k=-(d_{E_k}^*)^2)  \}.$$
In the problem described in this paper, to ensure safety, both the initial positions and target positions of the multi-agent system satisfy the following conditions:
\begin{equation}\label{cs1}
	\Vert q_{E_k}(0)\Vert>d_s,\quad d_{E_k}^*>d_s,\quad k=1,\cdots,M^t.
	\end{equation}
Next, we illustrate that if the condition \eqref{cs1} holds, then the closed-loop system \eqref{c1} asymptotically converges to the set $$S_1=\{(\bar q,\bar p,e)\vert \bar p=\textbf{0}_{nN},e_k=0 \},$$
 and the whole running process satisfies $\Vert q_{E_k}(t)\Vert\geq d_s$, where $t\geq 0$. 
 \begin{figure}
 	\begin{center}
 		\includegraphics[width=5.5cm]{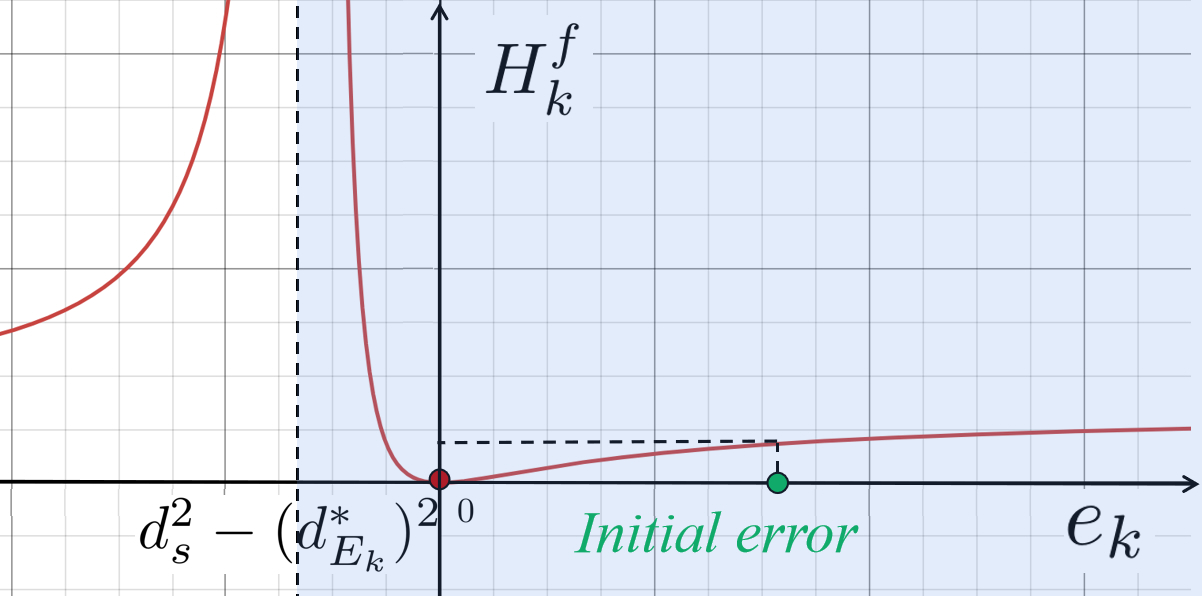}    % The printed column width is 8.4 cm.
 		\caption{The graph of the Hamiltonian function $H_k^f(e_k)$.} 
 		\label{f4}
 	\end{center}
 \end{figure}
Assume the initial values of the multi-agent systems are $\bar p(0)=p(0)-Mv^*$ and $e_k(0)> d_s^2-(d_{E_k}^*)^2$. 
As shown in Fig.~\ref{f4}, there exists a singularity $e_k+(d^*_{E_k})^2=d_s^2$, such that if $e_k(t)\rightarrow (d_s^2-(d_{E_k}^*)^2)^{+}$, then the Hamiltonian function
$$H_k^f=\frac{1}{4}\alpha_k(\frac{1}{e_k+(d_{E_k}^*)^2-d_{s}^2}\! -\! \frac{1}{(d_{E_k}^*)^2-d_{s}^2})^2\rightarrow +\infty.$$
It means that the candidate Lyapunov function $H^F\rightarrow +\infty$. Since $H^F\geq 0$ where the equality holds only when $\bar p=\textbf{0}_{nN}$, $e_k=0$, and $\dot{H}^F\leq 0$ where the equality holds only when $\bar p=\textbf{0}_{nN}$, we obtain a positively invariant set of the closed-loop system \eqref{c1} as follows:
$$S_2=\{(\bar q,\bar p,e)\vert e_k>d_s^2-(d_{E_k}^*)^2\}.$$
Hence, as long as the initial point of the system satisfies $\Vert q_{E_k}(0)\Vert>d_s$, then the trajectory of the system \eqref{c1} will lie within the set $S_2$, meaning that $\Vert q_{E_k}(t)\Vert>d_s$ holds and collisions between agents will not occur. 
Moreover, since the equilibrium $e_k=-(d_{E_k}^*)^2<d_s^2-(d_{E_k}^*)^2$, we conclude that this equilibrium will not achieved with the given initial states, which implies that the closed-loop system \eqref{c1} asymptotically converges to the set $S_1$. \qed
\end{pf}

%%
%% \begin{thm} ... \end{thm}            % Theorem
%% \begin{lem} ... \end{lem}            % Lemma
%% \begin{claim} ... \end{claim}        % Claim
%% \begin{conj} ... \end{conj}          % Conjecture
%% \begin{cor} ... \end{cor}            % Corollary
%% \begin{fact} ... \end{fact}          % Fact
%% \begin{hypo} ... \end{hypo}          % Hypothesis
%% \begin{prop} ... \end{prop}          % Proposition
%% \begin{crit} ... \end{crit}          % Criterion

\section{Simulation results}
%In this section, a numerical example is given to verify the effectiveness of the proposed controller.
Consider $3$ agents modeled under the PH framework as \eqref{e1}, where $q_i(t), p_i(t)\in\mathbb{R}^2$, the parameter matrices are given as $m_i=1$, $D_i=\text{diag}(1,0.8)$, $i=1,2,3$. The initial positions of each agent is given as $q_1(0)=\text{col}(0,2)$, $q_2(0)=\text{col}(7,0)$, $q_3(0)=\text{col}(4,1)$. Assume the multi-agent system starts from a stationary state, i.e., with an initial velocity of zero, and the control objective is to form an equilateral triangular formation with a side length of $4$ and maintain this formation while advancing at a desired constant velocity $v^*_1=\text{col}(0.5,0.5)~\text{m/s}$. The safe distance is predefined as  $d_s=1~\text{m}$.

The simulation results by using the controller proposed in \citep{Nb1} are given in Fig.~\ref{fl1}-\ref{fl2}, where the target velocity and the desired formation are achieved. However, if we let the edge $E_1=(1,2)$, $E_2=(2,3)$, $E_3=(1,3)$, then from Fig.~\ref{fl2} we can find that $d_{E_1}(0.86)=0.977<d_s$, $d_{E_2}(0.66)=0.971<d_s$, $d_{E_3}(0.74)=0.968<d_s$, which means collision occurs between agents.
\begin{figure}
\begin{center}
\includegraphics[width=6cm]{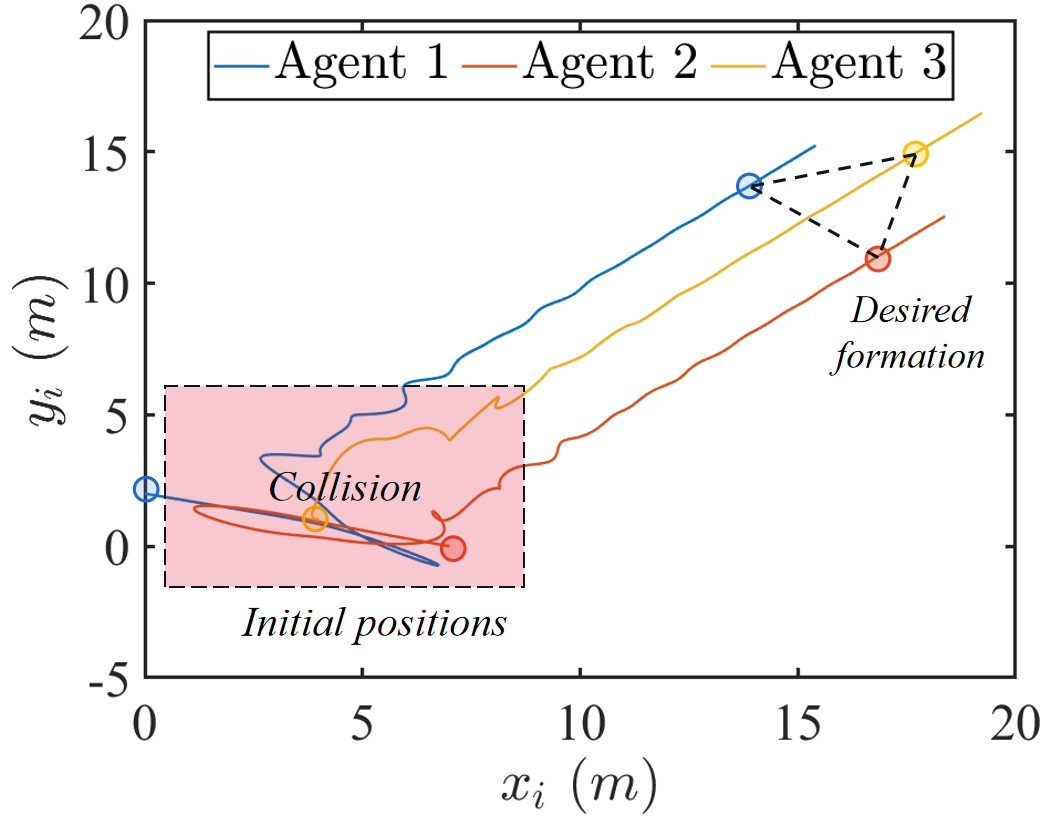}    % The printed column width is 8.4 cm.
\caption{The distance error $e_k$ by using controller proposed in \citep{Nb1}.} 
\label{fl1}
\end{center}
\end{figure}
\begin{figure}
	\begin{center}
		\includegraphics[width=6cm]{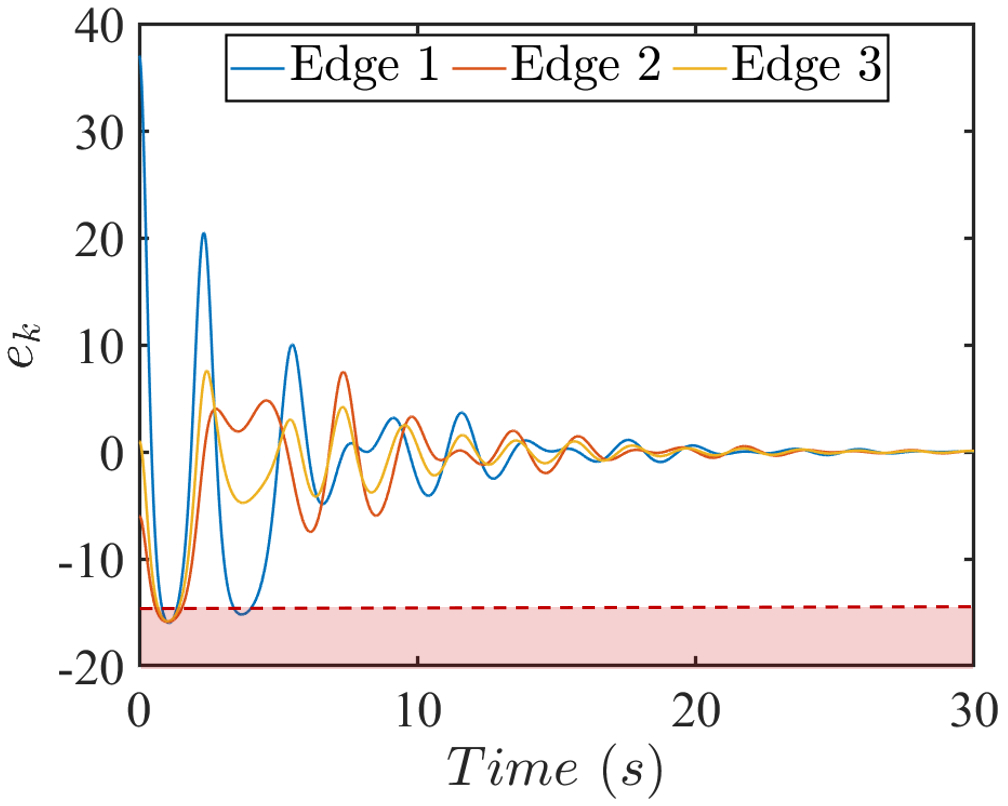}    % The printed column width is 8.4 cm.
		\caption{The system trajectory by using controller proposed in \citep{Nb1}.} 
		\label{fl2}
	\end{center}
\end{figure}
\begin{figure}
	\begin{center}
		\includegraphics[width=6cm]{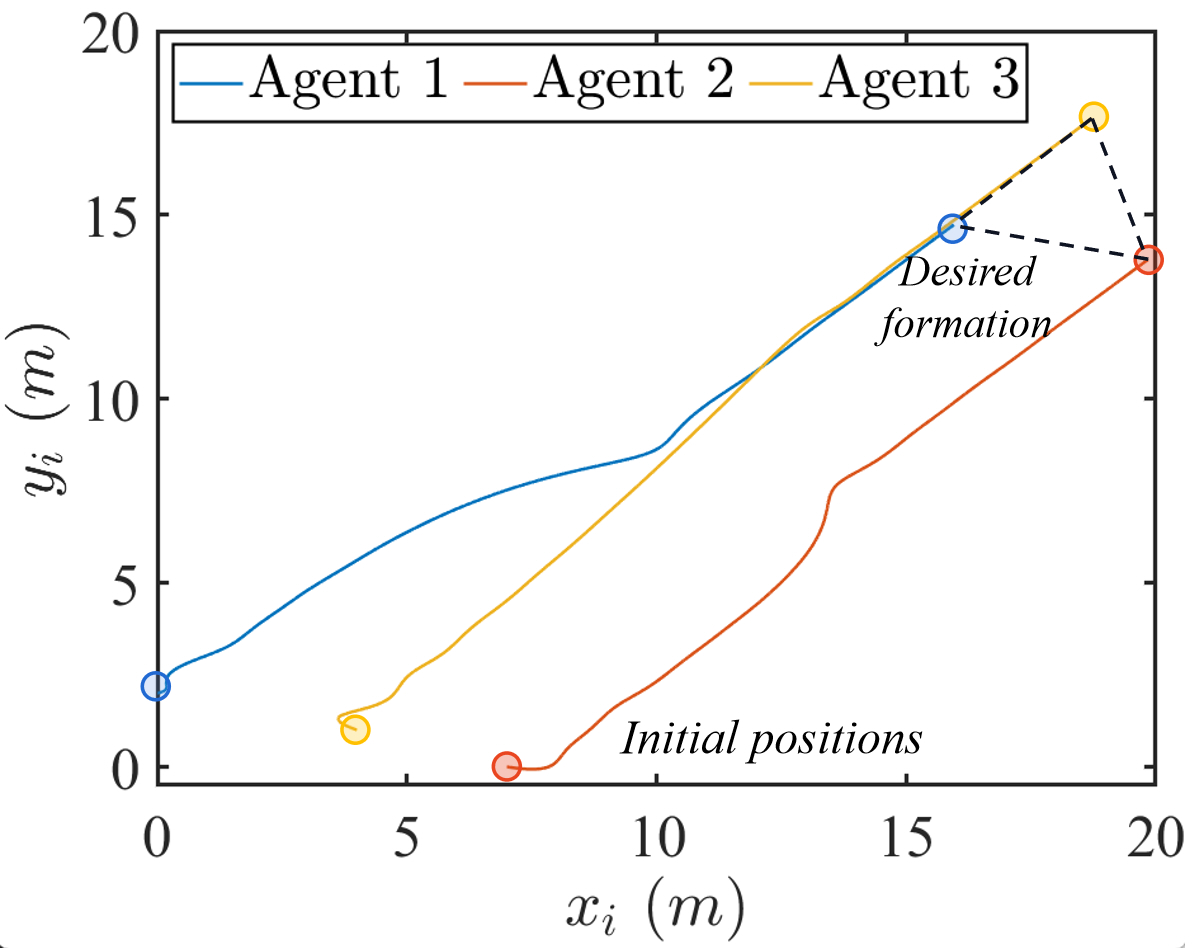}    % The printed column width is 8.4 cm.
		\caption{The system trajectory by using controller \eqref{u2}.} 
		\label{fs1}
	\end{center}
\end{figure}
\begin{figure}
	\begin{center}
		\includegraphics[width=6cm]{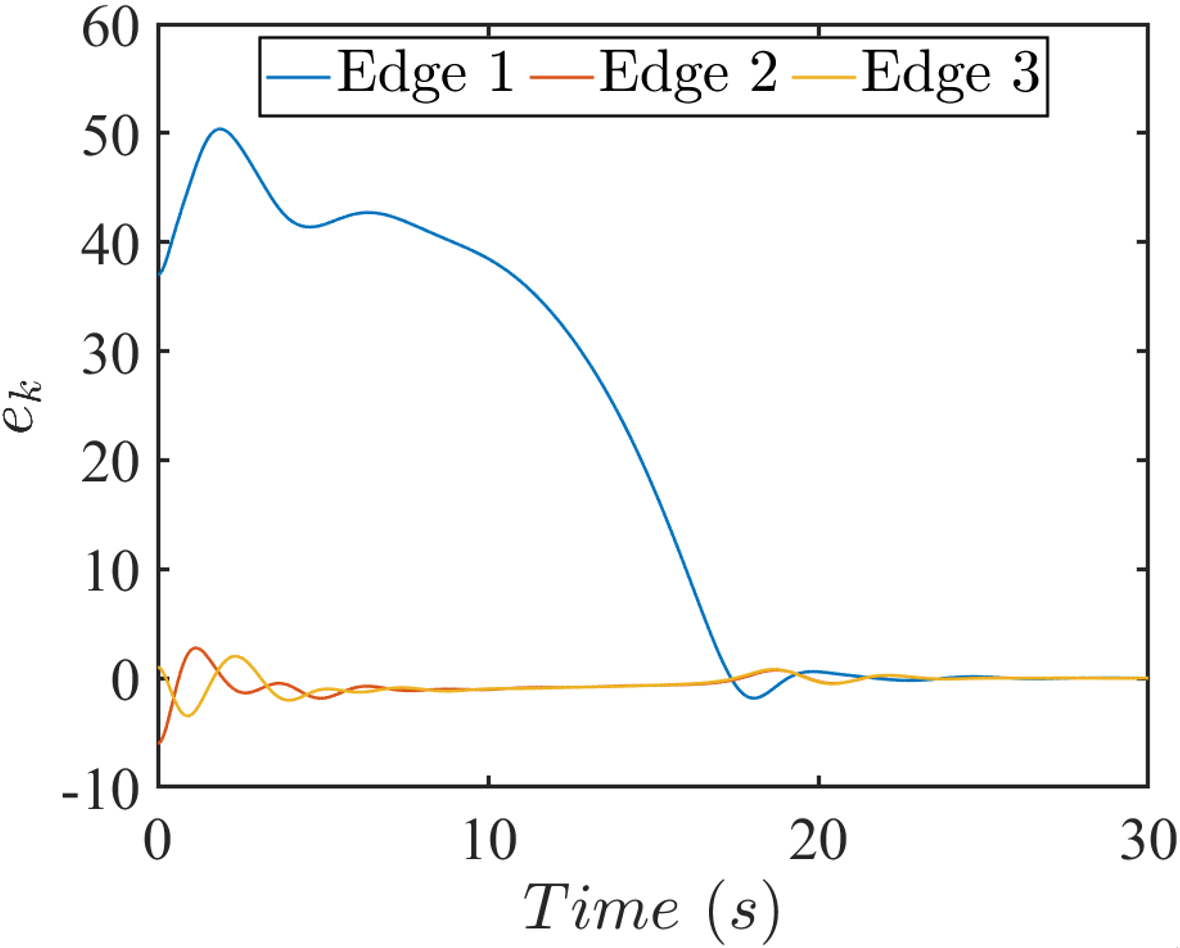}    % The printed column width is 8.4 cm.
		\caption{The trajectory of error $e_k$ by using controller \eqref{u2}.} 
		\label{fs2}
	\end{center}
\end{figure}
\begin{figure}
	\begin{center}
		\includegraphics[width=7cm]{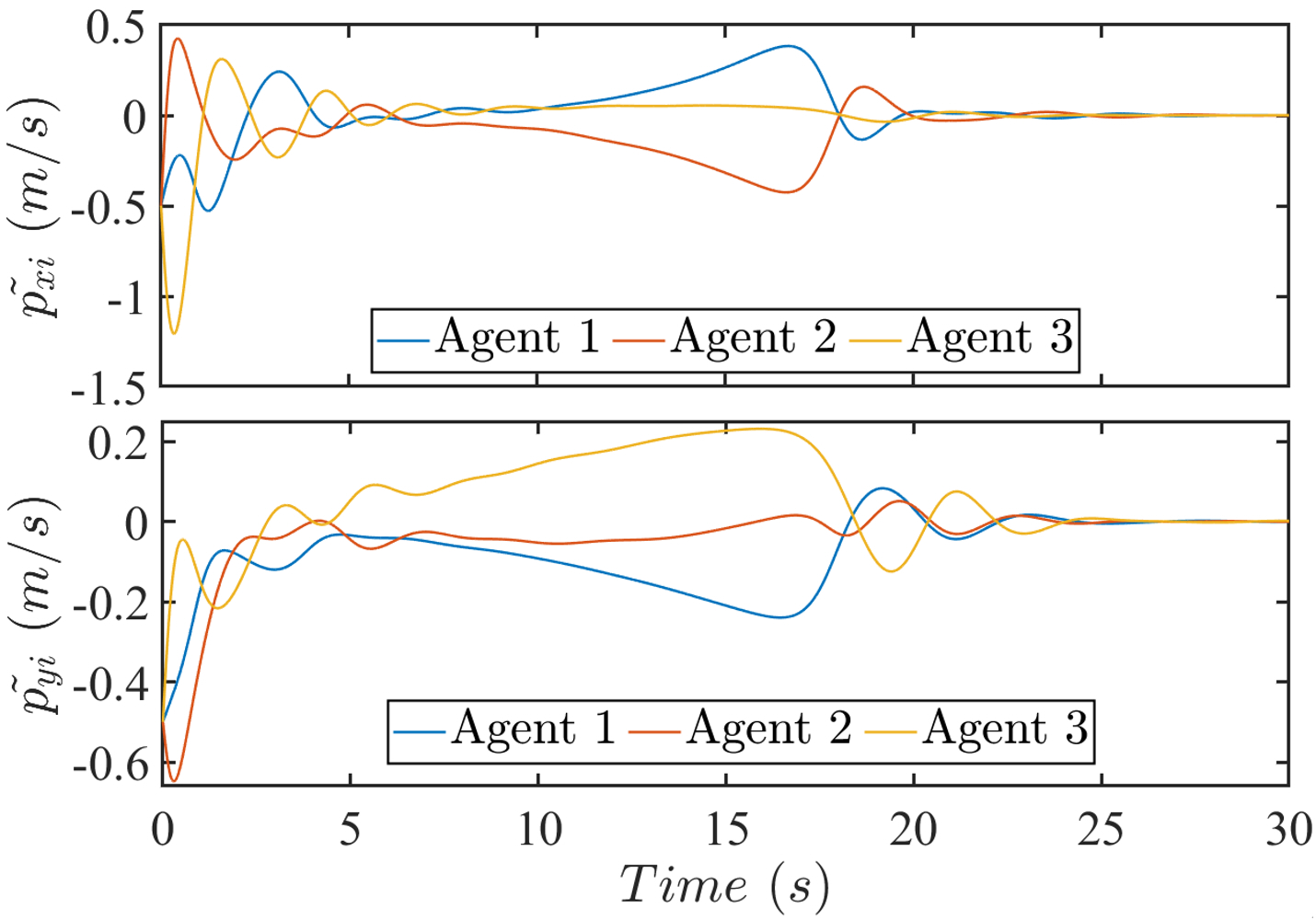}    % The printed column width is 8.4 cm.
		\caption{The trajectory of error $\bar p_i$ by using controller \eqref{u2}.} 
		\label{fs3}
	\end{center}
\end{figure}
\begin{figure}
	\begin{center}
		\includegraphics[width=6cm]{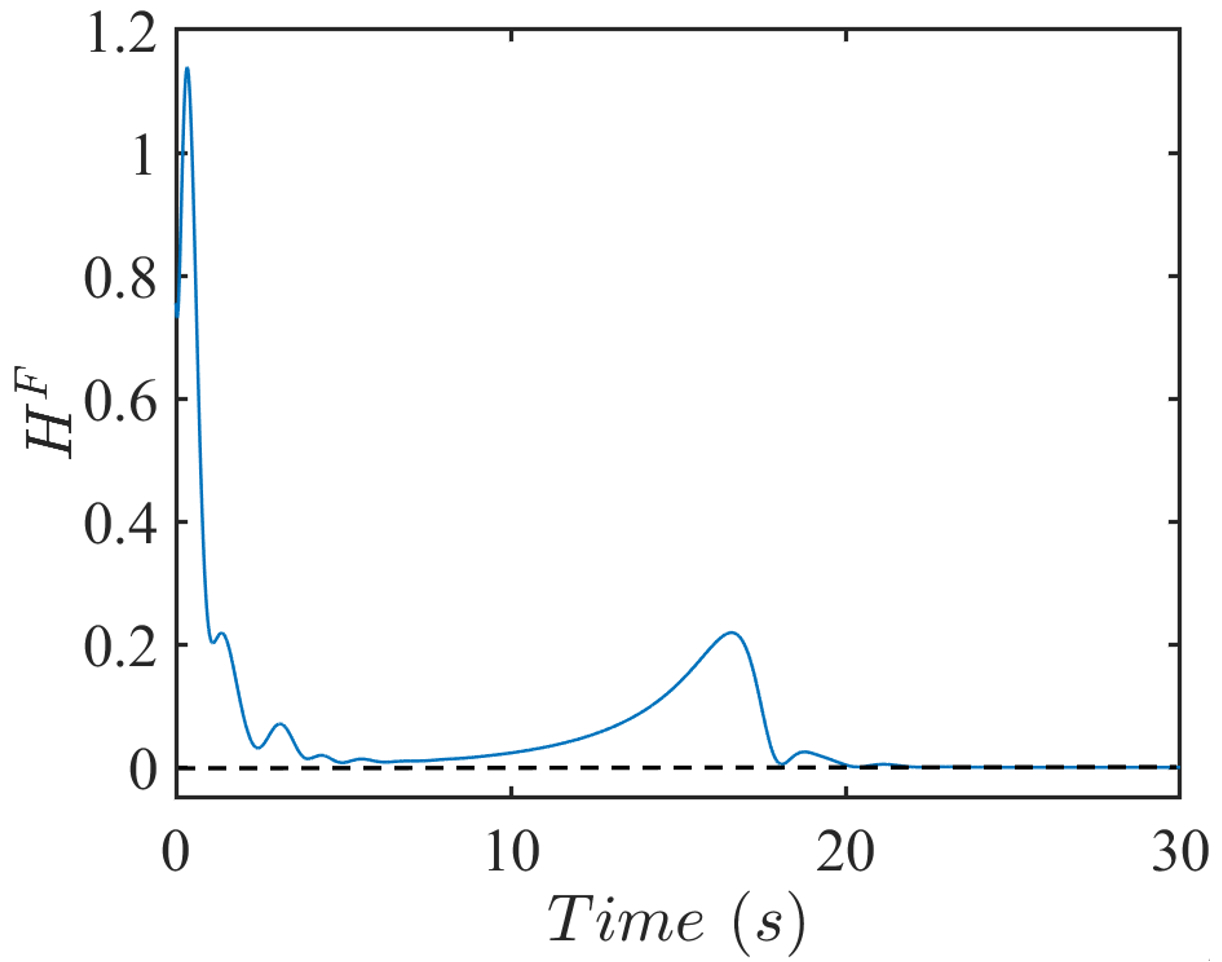}    % The printed column width is 8.4 cm.
		\caption{The curve diagram of $H^F$ by using controller \eqref{u2}.} 
		\label{fs4}
	\end{center}
\end{figure}

By giving the controller parameters as $\alpha_{E_k}=5$, $k=1,2,3$, then the simulation results are shown in Fig.~\ref{fs1}-\ref{fs4}. As shown in Fig.~\ref{fs1}-\ref{fs2}, under the action of controller \eqref{u2}, the desired formation is achieved. Since $e_k(t)=\Vert q_{E_k}(t)\Vert^2-(d_{E_k}^*)^2>d_s^2-(d_{E_k}^*)^2=1-16=-15$ holds with $t\geq 0$, we conclude that the collision never occurs between agents. Fig.~\ref{fs3} indicates that the multi-agent systems have tracked the target velocity, and Fig.~\ref{fs4} shows that the Hamiltonian function converges to $0$, which implies the asymptotically stability of the multi-agent systems.

These simulation results illustrate that the multi-agent systems can achieve the velocity tracking and formation objectives without collision by using \eqref{u2}.

\section{Conclusion}
In this work, a novel control framework for PH multi-agent systems formation and collision avoidance has been presented. By designing an attraction-only potential field and integrating it with safety barrier constraints, the proposed strategy effectively eliminates the risk of local minima and guarantees collision-free maneuvers. Stability of the overall system was established via LaSalle's invariance principle. The simulation studies confirm that the controller successfully achieves accurate formation tracking while maintaining safe distances among all agents. Future work will focus on extending the approach to scenarios with limited communication ranges and uncertain dynamics.

%\begin{ack}
%Place acknowledgments here.
%\end{ack}

\bibliography{ifacconf}             % bib file to produce the bibliography
                                                     % with bibtex (preferred)
                                                   
%\begin{thebibliography}{xx}  % you can also add the bibliography by hand

%\bibitem[Able(1956)]{Abl:56}
%B.C. Able.
%\newblock Nucleic acid content of microscope.
%\newblock \emph{Nature}, 135:\penalty0 7--9, 1956.

%\bibitem[Able et~al.(1954)Able, Tagg, and Rush]{AbTaRu:54}
%B.C. Able, R.A. Tagg, and M.~Rush.
%\newblock Enzyme-catalyzed cellular transanimations.
%\newblock In A.F. Round, editor, \emph{Advances in Enzymology}, volume~2, pages
%  125--247. Academic Press, New York, 3rd edition, 1954.

%\bibitem[Keohane(1958)]{Keo:58}
%\newblock \emph{Power and Interdependence: World Politics in Transitions}.
%\newblock Little, Brown \& Co., Boston, 1958.

%\bibitem[Powers(1985)]{Pow:85}
%\newblock Is there a way out?
%\newblock \emph{Harpers}, pages 35--47, June 1985.

%\bibitem[Soukhanov(1992)]{Heritage:92}
%\newblock \emph{{The American Heritage. Dictionary of the American Language}}.
%\newblock Houghton Mifflin Company, 1992.

%\end{thebibliography}
                                                                   % in the appendices.
\end{document}